\documentclass[11pt,reqno]{amsart}
\usepackage{dsfont}
\usepackage{amssymb}
\usepackage{amsfonts}
\usepackage{amsthm}
\usepackage{amsmath}
\usepackage{array}
\usepackage{stackrel}
\usepackage{a4wide}
\usepackage{amssymb}
\usepackage[font=footnotesize]{subfig}
\usepackage{color}
\usepackage{ifthen} 
\provideboolean{shownotes} 
\setboolean{shownotes}{true} 
\newcommand{\margnote}[1]{
\ifthenelse{\boolean{shownotes}}%
{\marginpar{\raggedright\tiny\texttt{#1}}}%
{}%
}
\newcommand{\hole}[1]{
\ifthenelse{\boolean{shownotes}}%
{\begin{center} \fbox{ \rule {.25cm}{0cm}
\rule[-.1cm]{0cm}{.4cm} \parbox{.85\textwidth}{\begin{center}
\texttt{#1}\end{center}} \rule {.25cm}{0cm}}\end{center}}
{}
}

\newtheorem{lemma}{Lemma}[section]
\newtheorem{theorem}[lemma]{Theorem}

\theoremstyle{definition}
\newtheorem{definition}[lemma]{Definition}
\theoremstyle{definition}
\newtheorem{remark}[lemma]{Remark}
\theoremstyle{definition}

{\catcode `\@=11 \global\let\AddToReset=\@addtoreset}
\AddToReset{equation}{section}

\newcommand{\veps}{{\varepsilon }}

\newcommand{\rR}{\mathbb R}
\newcommand{\rT}{\mathbb T}

\newcommand{\eps}{\varepsilon}

\newcommand{\dx}{\mathop{{\rm div}_{x}}}

\newcommand{\brho}{\bar \rho}

\newcommand{\barm}{\bar m}

\renewcommand{\div}{\operatorname{div}_x}

\newcommand{\del}{\partial}

\title[Weak-Strong Uniqueness for Euler-Korteweg]
{Stability properties of the Euler-Korteweg system \\
with nonmonotone pressures}

 \author{Jan Giesselmann} 
 \address[Jan Giesselmann]{\newline
  Institute of Applied Analysis and Numerical Simulation\newline
 University of Stuttgart\newline
 Pfaffenwaldring 57\newline
 D-70563 Stuttgart\newline
   Germany} 
 \curraddr{}
 \email{{jan.giesselmann@mathematik.uni-stuttgart.de}} 
 \thanks{
JG thanks the Baden-W{\"u}rttemberg foundation for support via the project 'Numerical Methods 
for Multiphase Flows with Strongly Varying Mach Numbers'. }

\author{Athanasios E. Tzavaras}
\address[Athanasios E. Tzavaras]{\newline 
Computer, Electrical, Mathematical Sciences \& Engineering Division
\newline
King Abdullah University of Science and Technology (KAUST)
\newline
Thuwal,  Saudi Arabia
}
\email{athanasios.tzavaras@kaust.edu.sa}

\begin{document}
\maketitle

\begin{center}
{\sf To Peter Markowich with friendship and admiration}
\end{center}

\begin{abstract}
We establish a relative energy framework for the Euler-Korteweg system with non-convex energy.
This allows us to prove weak-strong uniqueness  and to show convergence to a Cahn-Hilliard system in the large friction limit.
We also use relative energy to show that solutions of Euler-Korteweg with convex energy converge to solutions of the Euler system 
in the vanishing capillarity limit, as long as the latter admits sufficiently regular strong solutions.
\end{abstract}

\section{Introduction}
The isothermal Euler-Korteweg (EK) system is a well-known model for the description of  liquid-vapor flows. It 
contains as a special case the system of quantum hydrodynamics obtained
by applying the Madelung transform to the Schroedinger equation\cite{GM97}. It consists of the compressible Euler equations augmented 
to contain dispersive terms modeling capillarity.
It goes back to the 19th century but was derived using modern thermodynamic methods in \cite{DS85}.
For a review on its analytical and numerical treatment see \cite{BDDJ05,B10}.
Well-posedness and stability results for (local in time) smooth solutions can be found in \cite{BDD07}, 
and are valid for a large class of (capillarity) constitutive functions, and pressure laws.
By contrast, the issue of existence of global weak solutions is widely open at present.

We will consider the model in the form:
\begin{equation} \label{eq:EK}
    \begin{split}
	\rho_{t} +\dx (\rho u) &=0  \\
    (\rho u)_{t} + \dx (\rho u \otimes u)
	&=  - \rho \nabla_x \Big( h'(\rho) + \frac{\kappa'(\rho)}{2} |\nabla_x \rho|^2  -\dx (\kappa(\rho) \nabla_x \rho)  \Big )\, ,  \\
    \end{split}
\end{equation}
where $\rho \geq 0$ is the density, $u\in \mathbb{R}^d$ the velocity, $m=\rho u$  the momentum, 
$h=h(\rho)$ is the energy density, and
$\kappa=\kappa(\rho)>0$ is the coefficient of capillarity.

Note that \eqref{eq:EK} can also be written in conservative form as 
\begin{equation} \label{eq:EKcon}
    \begin{split}
	\rho_{t} +\dx (\rho u) &=0  \\
    (\rho u)_{t} + \dx (\rho u \otimes u)
	&=  \dx S\, ,  \\
    \end{split}
\end{equation}
where $S$ is the Korteweg stress tensor,
\begin{equation}\label{stress}
 S:= \Big[-p(\rho)  - \frac{\rho  \kappa'(\rho)+ \kappa(\rho)}{2}|\nabla_x \rho|^2  + \dx ( \rho \kappa(\rho) \nabla_x \rho) \Big ]  \mathbb{I} - \kappa(\rho)\nabla_x \rho \otimes \nabla_x \rho
\end{equation}
 $\mathbb{I}$ denotes the identity matrix in $\mathbb{R}^{d\times d}$ and the (local) pressure is defined as
\begin{equation}
 p(\rho)= \rho h'(\rho) - h(\rho) \, , \quad p'(\rho) = \rho h''(\rho) \, .
\end{equation}
Strong solutions of \eqref{eq:EK} satisfy the balance of total (internal and kinetic) energy,
\begin{multline} \label{eq:energy}
\partial_t \left  (\frac{1}{2} \frac{|m|^2}{\rho} + h(\rho) + \frac{\kappa(\rho)}{2} |\nabla_x \rho|^2 \right )  \\
  + \dx  \left ( \frac{1}{2}m\frac{|m|^{2}}{\rho^{2}} + 
     m\Big( h'(\rho)+ \frac{\kappa'(\rho)}{2} |\nabla_x \rho|^2 - \dx (\kappa(\rho) \nabla_x \rho) \Big)    + \kappa(\rho) \nabla_x\rho\dx m 
          \right )=  0.
\end{multline}

The purpose of this paper is to study some consequences of the relative energy framework 
for the Euler-Korteweg system presented in \cite{GLT} and to develop extensions for fluids with non-monotone pressures.
In \cite{GLT} a relative energy framework for general fluid dynamics systems arising from a Hamiltonian structure is developed.
The framework is conceptually related to the relative entropy computations of Dafermos and DiPerna \cite{Daf79,Daf79b,Dip79}
well known in the theory of hyperbolic conservation laws. However, the former emerges from a variational structure while
the latter emerges from the Clausius-Duhem inequality of
thermomechanical  systems.

One key goal is to extend the stability implications of relative energy to Euler-Korteweg systems
with non-convex (local) energies.
We will achieve that in the case of constant capillarity, which is a typical setting when \eqref{eq:EK} is used to model liquid-vapor flows.
In the study of liquid-vapor flows the energy is convex on the majority of the state space and there is only a small interval, 
usually called the elliptic region, in which it is concave.
Thus, we restrict to energies obeying the structure:
\begin{equation}
\label{decomp_h} 
h(\rho) = h_\gamma(\rho) + e(\rho) := c \rho^\gamma + e(\rho)
\end{equation}
with $c>0$, $\gamma > 1$ and $e(\rho) \in C_c^\infty(0,\infty)$ smooth and compactly supported.
In \eqref{decomp_h} the (local) internal energy consists of a $\gamma$-law part $h_\gamma$ and 
another (localized in state space) part  $e$ containing the non-convexity.

The extension of the relative energy framework to non-convex energies suggest to consider the issues of  weak-strong uniqueness 
and of large friction limits to systems admiting non-convex energies.
The question of weak-strong uniqueness for the Euler-Korteweg system was addressed in \cite{DFM15,GLT}  for fluids with monotone pressure laws.
Such results were not known to apply to weak solutions for systems of liquid-vapor flows.
The present note closes this gap, i.e., we show weak-strong uniqueness for the Euler-Korteweg model with non-convex energy.

This modified relative energy framework for non-convex energies also allows to carry out the  large friction limit 
from the Euler-Korteweg system  to the Cahn-Hilliard equation again for non-convex local energy.
For convex energies, such limits were established in \cite{LT16} following a general strategy for convergence from Hamiltonian systems 
with friction to gradient flows.

Our second objective is to study the vanishing capillarity limit of \eqref{eq:EK} in the smooth regime,
restricting  now to convex energy densities. 
We show that solutions to \eqref{eq:EK} converge as $\kappa(\rho) \rightarrow 0$ to solutions of the 
associated compressible  Euler system (with $\kappa \equiv 0$) as long as the  limiting Euler system admits smooth solutions. 
We are not aware of such a general convergence result in the smooth regime, and the proof is a remarkably simple consequence
of the relative energy framework.

Beyond the smooth regime the effect of dispersion is known
(at the level of integrable systems) to induce oscillations and the behavior in the small capillarity limit of shocks is not at all understood. 
On the other hand,  the inclusion of viscosity has a stabilizing effect, but again the zero viscosity-capillarity limit is a very subtle process.
The reader is referred to \cite{CH13, GL15} for existence results on the Navier-Stokes Korteweg systems, and
to \cite{BL03} for the intricacies of its use as a selection criterion to the issue of shock admissibility.

The outline of the rest of the paper is as follows: In section \ref{sec:sre} we recall notions of dissipative and conservative weak solutions as well as relative energy computations for the Euler-Korteweg system.
Section \ref{sec:wsnce} deals with removing the terms from the relative energy which are related to $e$. This allows us 
to prove weak-strong uniqueness.
This version of relative energy is used for studying the large friction limit in section \ref{sec:lf}.
Section \ref{sec:vcl} is devoted to establishing the vanishing capillarity limit for convex energies.

\section{Relative Energy}\label{sec:sre}
To keep this paper self contained, we recall notions of weak solutions for \eqref{eq:EK} and standard relative energy arguments in this section following the exposition in \cite{GLT}.
For simplicity, we focus on periodic solutions, defined in $\rT^d$  the $d$-dimensional flat torus.
Extending our results to finite domains with Dirichlet or Neumann boundary conditions is not straightforward. 
 Many of the computations performed here can be carried out for solutions in $\mathbb{R}^d$.
However,  to obtain similar results one would be restricted to adiabatic coefficients $\gamma \geq 2$,  since the estimates 
in the range $1 < \gamma < 2$ make frequent use of Poincar{\'e}'s inequality.

\medskip

We recall:

\begin{definition} \label{def:wksol}
(i) A function $( \rho,  m)$ with  $\rho \in C([0, \infty) ; L^1 ( \rT^d ) )$, $m \in C  \big (  [0, \infty) ;   L^1(\rT^d, \rR^d)  \big )$,
$\rho \ge 0$,  is a weak
solution of \eqref{eq:EK}, if $\frac{m \otimes m}{\rho}$, $S \in L^1_{loc}  \left (  (0, \infty) \times \rT^d ) \right )^{d \times d}$,  and $(\rho, m)$ satisfy
\begin{equation}\label{eq:wksol}
\begin{aligned}
- \iint \rho \psi_t + m \cdot \nabla_x \psi dx d t  = \int \rho (0 , x) \psi (0 , x) dx \, ,  \qquad  \forall \psi \in C^1_c \left ( [0, \infty) ; C^1 (\rT^d) \right )\, ;
\\
- \iint m \cdot \varphi_t + \frac{m \otimes m}{\rho} : \nabla_x \varphi - S : \nabla_x \varphi \,  dx dt = \int m(0,x) \cdot \varphi(0,x) dx  \, ,  \qquad \qquad 
\\
 \forall  \ \varphi  \in C^1_c \left ( [0, \infty) ;  \big ( C^1 (\rT^d)  \big )^d \right ) \, .
\end{aligned}
\end{equation}

\medskip\noindent
(ii) If,  in addition, $\frac{1}{2} \frac{|m|^2}{\rho} + h(\rho) + \frac{\kappa(\rho)}{2} |\nabla_x \rho|^2 \in C([0, \infty) ; L^1 ( \rT^d ) )$ and it satisfies
\begin{equation}
 \label{eq:disws}
 \begin{aligned}
  - \iint \Big ( \frac{1}{2} \frac{|m |^2}{\rho} + h(\rho) &+ \frac{\kappa(\rho)}{2} |\nabla_x \rho|^2  \Big ) \dot\theta(t) \, dx dt
\\
 &\le   \int \left ( \frac{1}{2} \frac{|m |^2}{\rho } + h(\rho) + \frac{\kappa(\rho)}{2} |\nabla_x \rho|^2  \right ) \Big |_{t=0}  \theta(0)dx \, ,
\\
& \forall \ \theta  \in W^{1, \infty} [0, \infty) \, , \theta \ge 0 \, , \ \mbox{compactly supported on $[0, \infty)$},
 \end{aligned}
\end{equation}
then $(\rho, m)$ is called a dissipative weak solution. 

 \medskip\noindent
(iii) By contrast,
if $\frac{1}{2} \frac{|m|^2}{\rho} +h(\rho) + \frac{\kappa(\rho)}{2} |\nabla_x \rho|^2  \in C([0, \infty) ; L^1 ( \rT^d ) )$ and it satisfies \eqref{eq:disws}
as an equality, then $(\rho, m)$  is called a conservative weak solution.
\end{definition}

Depending on how the solutions arise one might be inclined to use either  \emph{conservative} or \emph{dissipative} weak solutions.
The appropriate notion  depends on how solutions emerge:  if they emerge as vanishing viscosity
limits of models in fluid mechanics then the use of dissipative solutions is advisable; by contrast, if they arise 
from the Schroedinger equation as in the case of the QHD system the use of conservative solutions might be more appropriate.
In any case, our analysis covers both eventualities.

In the sequel we consider weak solutions with finite mass and finite energy, i.e., we place the assumption
\begin{itemize}
\item[\textbf{(H)}] $( \rho,  m)$ is  a dissipative (or conservative) weak periodic solution of \eqref{eq:EK} with $\rho \ge 0$
in the sense of Definition \ref{def:wksol}, and 
\begin{align}
\sup_{t\in (0,\infty)} \int_{\rT^d}  \rho \, dx &\le K_1 < \infty\, ,
\label{hypCauchyK1}
\\
 \sup_{t\in (0,\infty)} \int_{\rT^d}
 \frac{1}{2} \frac{|  m |^2}{ \rho} + h(\rho) + \frac{\kappa(\rho)}{2} |\nabla_x \rho|^2  \, dx
 &\le K_2 < \infty \, .
 \label{hypCauchyK2}
\end{align}
\end{itemize}

\begin{definition}[Strong Solution]
 We call $(\rho,\rho u)$ a strong solution of \eqref{eq:EK} on $[0,T) \times \rT^d$ provided
 \begin{align}
  \rho \in C^0([0,T), C^3(\rT^d)) \cap C^1([0,T), C^1(\rT^d))\\
  u \in C^0([0,T), C^2(\rT^d,\rR^d)) \cap C^1([0,T), C^0(\rT^d,\rR^d))
 \end{align}
and \eqref{eq:EK} is satisfied in a point-wise sense.
\end{definition}
To deal with the capillary part of the energy we define
\begin{equation}
 F(\rho,q) := \frac{\kappa(\rho)}{2} |q|^2 \quad \text{ for  any } \rho>0, q \in \mathbb{R}^d.
\end{equation}

Following Dafermos and DiPerna \cite{Daf79, Dip79} we define relative quantities comparing two different states.
They are given by the value at one state minus the first order Taylor expansion around the other state.
The relative potential energy consists of two parts:
\begin{equation}\label{def:rel}
 \begin{split}
  h(\rho  | \, \bar \rho ) &:= h(\rho) - h(\brho) - h'(\brho) (\rho - \brho),\\
  F(\rho, q| \brho , \bar q) &:= F(\rho, q) - F( \brho , \bar q) - \frac{\partial F}{\partial \rho}(\brho, \bar q) (\rho - \brho)
  - \frac{\partial F}{\partial q}(\brho, \bar q) (q - \bar q).
 \end{split}
\end{equation}
For the  kinetic energy $K(\rho ,m) = \tfrac{1}{2} \tfrac{|m|^2}{\rho}$ we proceed analogous to $F$ and 
\begin{equation}
 K(\rho, m | \brho, \bar m) =  \frac{1}{2} \rho \Big | \frac{m}{\rho} -  \frac{\bar m}{\bar\rho}\Big | ^2,
\end{equation}
see \cite{GLT} for details.

For the energy density at hand \cite[Theorem 3.2]{GLT} implies that for any strong solution $(\brho,\ \barm)$ of \eqref{eq:EK}
and any weak (dissipative or conservative) solution $(\rho,m)$ of \eqref{eq:EK} the following inequality is fulfilled for almost all $t \in [0,T):$
\begin{equation}
 \label{eq:RelEnKorGenFinalweak}
\begin{split}
& \int_{\rT^d} \left. \Big ( \frac{1}{2} \rho \Big | \frac{m}{\rho} -  \frac{\bar m}{\bar\rho}\Big | ^2 +  h(\rho  | \, \bar \rho \right. )+ F(\rho,\nabla_x \rho| \brho , \nabla_x \brho) \Big )dx \Big |_t\\
&\leq
 \int_{\rT^d} \left. \left ( \frac{1}{2} \rho \Big | \frac{m}{\rho} -  \frac{\bar m}{\bar\rho}\Big | ^2 +  h(\rho  | \, \bar \rho \right. )+ F(\rho,\nabla_x \rho| \brho , \nabla_x \brho) \right )dx \Big |_0
\\
 &\  - \int \!\!\int_{[0,t)\times\rT^d} \left [\rho \left( \frac{m}{\rho}-   \frac{\bar m}{\bar\rho}\right )\otimes \left( \frac{m}{\rho}-   \frac{\bar m}{\bar\rho}\right ) :\nabla_x  \left(    \frac{\bar m}{\bar\rho}\right )\right] \, ds dx
 \\
  &\  - \int \!\!\int_{[0,t)\times\rT^d} 
  \left[\dx \left (\frac{\bar m}{\bar\rho}\right ) \Big(  s(\rho, \nabla_x\rho\left | \bar \rho, \nabla_x\bar \rho \right. ) + p(\rho|\bar \rho)   \Big)  \right]dsdx
\\
 &\ 
 - \int \!\!\int_{[0,t)\times\rT^d} \left [ \nabla_x \left (\frac{\bar m}{\bar\rho} \right) : H(\rho, \nabla_x\rho\left | \bar \rho, \nabla_x\bar \rho \right. ) 
 + \nabla_x \dx \left (\frac{\bar m}{\bar\rho}\right ) \cdot r(\rho, \nabla_x\rho\left | \bar \rho, \nabla_x\bar \rho \right. )\right]dsdx\, ,
\end{split}
\end{equation}
where 
$ s(\rho , q,| \brho , \bar q) $, $r(\rho , q,| \brho , \bar q) $ and $H(\rho , q,| \brho , \bar q) $
are defined  in a way analogous to \eqref{def:rel}, based on the functions
\begin{equation}\label{eq:rsH}
 s(\rho , q)  =  \tfrac{1}{2} \big (\kappa(\rho)+ \rho \kappa'(\rho) \big ) | q |^2  \, , \quad
 r(\rho , q)  = \rho \kappa(\rho) q \, ,  \quad
  H(\rho , q) =  \kappa(\rho) q\otimes q  \, ,
\end{equation}
that appear in the Korteweg stress tensor  \eqref{stress}.

\section{Weak-strong uniqueness for non-convex energies}\label{sec:wsnce}
In this section we consider  energies possessing a decomposition of the form \eqref{decomp_h} and constant capillarities $\kappa(\rho)=C_\kappa >0.$ 
We will see that the gradient terms in the energy  (which express nonlocal contributions to the energy) allow to compensate for the non-convexity 
of the local part of the energy density. The energy estimates dictate that  a weak solution $(\rho, m)$ has the regularity:
\begin{equation}
\label{bounds}
\sup_{ t \in (0, \infty)} \int_{\rT^d} \tfrac{1}{2} \frac{ |m|^2}{\rho} + c \rho^\gamma + \tfrac{C_\kappa}{2} | \nabla \rho |^2 \, dx < \infty \, , \quad 
\end{equation}
for $\gamma > 1$. The decomposition \eqref{decomp_h} implies
\[ h(\rho  | \, \bar \rho ) =  h_\gamma(\rho  | \, \bar \rho ) +  e(\rho  | \, \bar \rho ),\]
and it will be our goal  to remove the terms containing $e$ from the left hand side of \eqref{eq:RelEnKorGenFinalweak}.
In particular, we aim at providing an upper  bound for 
\[   
 \int_{\rT^d} e(\rho  | \, \bar \rho  ) dx \Big |_0 - \int_{\rT^d} e(\rho  | \, \bar \rho  )  dx \Big |_t.\]
 In \cite{Gie14,GLT} this was done for two strong solutions  by estimating $\partial_t e(\rho| \brho).$
 However, here we only assume that $\rho$ is a weak solution, so $\partial_t e(\rho| \brho)$ is not well defined.
 To overcome this technical issue we derive the corresponding estimates, on the level of a smooth approximating sequence $(\rho^n, m^n)$
 and then go to the limit on both sides of the estimate.
 
 To this end, we first extend $(\rho, m)$ to negative times by setting
\begin{equation}
 \rho(t,x)= \rho_0(x),\qquad m(t,x) = 0 \quad \text{ for } t< 0.
\end{equation}
Note that these extended functions weakly solve 
\begin{equation}
 \partial_t \rho + \dx m =0 \quad \text{ in } (-\infty, \infty) \times \rT^d.
\end{equation}
Let $\phi \in C_c^\infty(\rR,[0,\infty))$ with $\operatorname{supp}(\phi)\subset [0,1]$ and $\int_{\rR} \phi=1,$
then we define
\begin{equation}\label{def:mol}
 \begin{split}
  \phi^n(x)&:= n \phi(nx),\\
  \rho^n(t,x)&:= \int_{\rR} \int_{\rT^d} \phi^n(t-s) \phi^n(x-y) \rho(s,y) \, dy ds,\\
  m^n(t,x)&:= \int_{\rR} \int_{\rT^d} \phi^n(t-s) \phi^n(x-y) m(s,y) \, dy ds,
 \end{split}
\end{equation}
where for $x \in \rT^d$ we understand 
\[ \phi^n(x) = \prod_{i=1}^d \phi^n(x_i).\]

By construction $\rho^n \in C^\infty ((-\infty,\infty) \times \rT^d,[0,\infty))$,  $m^n \in C^\infty ((-\infty,\infty) \times \rT^d, \rR^d)$ and $\rho^n(0,\cdot) \longrightarrow \rho_0$ in $H^1(\rT^d)$.
Moreover, for any test function $ \psi \in C_c^\infty ((-\infty,\infty) \times \rT^d)$
we define
\begin{equation}
 \psi^n(s,y):= \int_{\rR} \int_{\rT^d} \phi^n(t-s) \phi^n(x-y) \psi(t,x) \, dx dt\, .
\end{equation}
Then, we have
\begin{equation}\label{eq:35}\begin{split}
 &\int \!\!\!\int \rho^n \partial_t \psi + m^n \cdot \nabla_x \psi dx dt
\\
 &=
  \int  \!\!\!\int \!\!\! \int \!\!\! \int \phi^n(t-s) \phi^n(x-y) \big( \rho(s,y) \partial_t \psi (t,x) + m(s,y) \cdot \nabla_x \psi(t,x)\big) dx dy ds dt
\\
 &=   \int \!\!\! \int  \!\!\!\int \!\!\! \int \partial_s  \big ( \phi^n(t-s) \big )  \phi^n(x-y)  \rho(s,y)  \psi (t,x) 
\\
 & \qquad \qquad \qquad \qquad \qquad + \phi^n(t-s) \nabla_y \big ( \phi^n(x-y) \big ) \cdot  m(s,y) \psi(t,x) \,  dx dy ds dt
\\
 &= \int \!\!\! \int \rho(s,y) \partial_s \psi^n(s,y) + m(s,y) \cdot  \nabla_y \psi^n(s,y) ds dy=0.
\end{split}\end{equation}
Equation \eqref{eq:35} implies (due to the regularity of $\rho^n,m^n$) that for all $n \in \mathbb{N}$ the following equation holds in a point-wise sense
\begin{equation}\label{eq:regmc}
 \partial_t \rho^n + \dx m^n =0 \quad \text{ in } (-\infty, \infty) \times \rT^d.
\end{equation}

Let $(\rho^n,m^n)$ and  $(\brho,\barm)$ solve \eqref{eq:EK}$_1$ point-wise, then we obtain
 \begin{align}&\partial_t e(\rho^n| \brho)   \\
 &= \partial_t \big(e(\rho^n) - e(\brho) - e'(\brho)(\rho^n - \brho) \big)\nonumber \\
& = e'(\rho^n) \partial_t \rho^n - e''(\brho) \partial_t \brho (\rho^n - \brho) - e'(\brho)\partial_t \rho^n\nonumber \\
& = - \big(e'(\rho^n) - e'(\brho)\big) \div(m^n) + \div(\bar m)  e''(\brho) (\rho^n - \brho)\nonumber  \\
 & =-\div (\barm) \big( e'(\rho^n) - e'(\brho) - e''(\brho)(\rho^n-\brho) \big)
  - \big( e'(\rho^n) - e'(\brho)\big) \big( \div(m^n) - \div( \barm) \big)\nonumber 
  \end{align}
 and thus
  \begin{multline}\label{eq:erelreg}
   \int_{\rT^d} e(\rho^n| \brho) \, dx \big|_t - \int_{\rT^d} e(\rho^n| \brho) \, dx \big|_0
   \\
   =- \int_0^t \! \int_{\rT^d} \div ( \barm) \big( e'(\rho^n) - e'(\brho) - e''(\brho)(\rho^n-\rho) \big) \, dx ds
   \\ +
   \int_0^t \! \int_{\rT^d} \big( e''(\rho^n)\nabla_x \rho^n - e''(\brho)\nabla_x \brho\big) \cdot \big( m^n - \barm \big)\, dx ds .
  \end{multline}

From \eqref{bounds} we see that $\rho \in L^\gamma ((0,T)\times \rT^d)$, $\nabla \rho \in L^2 ((0,T)\times \rT^d))$ and $m \in L^1 ((0,T)\times \rT^d)$.
Standard properties of convolution dictate 
\begin{equation}\label{eq:conv3}
\begin{split}
 \rho^n  &\longrightarrow \rho \quad \text{ in } L^\gamma ((0,T)\times \rT^d),
\\
  \nabla \rho^n  &\longrightarrow \nabla \rho \quad \text{ in } L^2 ((0,T)\times \rT^d),
\\
 m^n  &\longrightarrow m \quad \text{ in }  L^1((0,T)\times \rT^d) \, ,
\end{split}
\end{equation}
and along a subsequence (if necessary) 
\begin{equation}\label{eq:conv1}
\begin{split}
 (\rho^n, m^n) &\longrightarrow (\rho,m) \quad \text{ a.e. in } (0,t) \times \rT^d \, .
 \\
 \nabla \rho^n  &\longrightarrow \nabla \rho  \qquad \text{ a.e. in }  (0,t) \times \rT^d,
 \end{split}
\end{equation}

\begin{lemma}\label{lem:erel}
 Let $(\rho , m)$ and $(\brho,\barm)$ be a weak and a strong solution of \eqref{eq:EK} respectively, then,
   \begin{multline}\label{eq:erel}
   \int_{\rT^d} e(\rho| \brho) \, dx \big|_t - \int_{\rT^d} e(\rho| \brho) \, dx \big|_0
   \\
   =  - \int_0^t \! \int_{\rT^d} \div ( \barm) \big( e'(\rho) - e'(\brho) - e''(\brho)(\rho-\brho) \big) \, dx ds
   \\ 
  +  \int_0^t \! \int_{\rT^d} \big( e''(\rho)\nabla_x \rho - e''(\brho)\nabla_x \brho\big) \big( m - \barm \big)\, dx ds .
  \end{multline}
\end{lemma}

\begin{proof}
The mollification $(\rho^n,m^n)$ of $(\rho,m)$, defined in \eqref{def:mol},  satisfies \eqref{eq:regmc} and \eqref{eq:erelreg}.
We study the limit of \eqref{eq:erelreg} as $n \rightarrow \infty$. Since the function $e \in C_c^\infty(0,\infty)$ is compactly supported in $(0, \infty)$,
the function $e(\rho)$ and its derivatives are uniformly bounded. 
As a  straightforward application of \eqref{eq:conv1}, \eqref{eq:conv3} and the dominated convergence theorem, we obtain
\begin{equation}
 \int_{\rT^d} e(\rho^n| \brho) \, dx \big|_t - \int_{\rT^d} e(\rho^n| \brho) \, dx \big|_0 \longrightarrow  \int_{\rT^d} e(\rho| \brho) \, dx \big|_t - \int_{\rT^d} e(\rho| \brho) \, dx \big|_0
\end{equation}
and 
\begin{multline}
   \int_0^t\!\! \int_{\rT^d} \div ( \barm) \big( e'(\rho^n) - e'(\brho) - e''(\brho)(\rho^n-\brho) \big) \, dx ds\\
\longrightarrow   \int_0^t\!\! \int_{\rT^d} \div ( \barm) \big( e'(\rho) - e'(\brho) - e''(\brho)(\rho-\brho) \big) \, dx ds \, .
\end{multline}

The last objective is to show the convergence
   \begin{multline}
   \int_0^t \!\!\! \int_{\rT^d} \big( e''(\rho^n)\nabla_x \rho^n - e''(\brho)\nabla_x \brho\big) \big( m^n - \barm \big)\, dx ds
   \\ 
  \longrightarrow 
   \int_0^t \!\!\! \int_{\rT^d} \big( e''(\rho)\nabla_x \rho- e''(\brho)\nabla_x \brho\big) \big( m - \barm \big)\, dx ds \, .
  \end{multline}
This is split into four terms and the only term whose convergence presents new difficulties is to show
\begin{equation}\label{pr:erel2}
 e''(\rho^n)\nabla_x \rho^n m^n \longrightarrow  e''(\rho)\nabla_x \rho \, m \quad \text{ in }  L^1((0,T)\times \rT^d) \, .
\end{equation}

To prove \eqref{pr:erel2} we will use the following variant of the dominated convergence theorem:
If $\{ f_n \}$ and $\{ g_n \}$ are sequences that satisfy $f_n \to f$ a.e., $|f_n |  \le g_n$ and  $g_n \to g$ in $L^1$
then $f_n \to f$ in $L^1$. We apply that to the functions
$$
f_n = e''(\rho^n)\nabla_x \rho^n m^n \, , \quad g_n =  C | \nabla_x \rho^n | \,  \Bigg | \sqrt{\left( \frac{|m|^2}{\rho}\right)^n} \Bigg |  \, .
$$

Note first that, since $e''(\rho)$ is compactly supported in $(0, \infty)$, the vacuum region is avoided and 
 it follows from \eqref{eq:conv1} that
$$
e''(\rho^n)\nabla_x \rho^n m^n  \to e''(\rho) \nabla_x \rho \, m \quad \mbox{a.e.}
$$
Since $k(m,\rho) = \frac{|m|^2}{\rho}$ is convex, applying (the multivariate version of) Jensen's inequality we obtain
\begin{equation}
 \frac{ |m^n|^2}{\rho^n} (t,x) \leq \int_{-\infty}^\infty \int_{\rT^d} \phi^n(x-y)\phi^n(t-s) \frac{|m(s,y)|^2}{\rho(s,y)}\, dy ds =: \left( \frac{|m|^2}{\rho}\right)^n (t,x).
\end{equation}
Moreover, as $\frac{|m|^2}{\rho} \in L^1$, we have 
\begin{equation}
\label{eq:conv4}
\left( \frac{|m|^2}{\rho}\right)^n \longrightarrow  \frac{|m|^2}{\rho}  \quad \mbox{in} \;  L^1([0,T] \times \rT^d)  \, .
\end{equation}
Observe next that
$$
\begin{aligned}
 |  e''(\rho^n)\nabla_x \rho^n m^n | &= \left  | \sqrt{\rho^n} e''(\rho^n) \nabla_x \rho^n \frac{ m^n}{\sqrt{\rho^n}}  \right |
\\
&\le C |\nabla_x \rho^n| \sqrt{ \frac{ |m^n|^2 } {\rho^n} }
\\
&\le C |\nabla_x \rho^n| \sqrt{ \left( \frac{|m|^2}{\rho}\right)^n  }
\end{aligned}
$$
and, on account of \eqref{eq:conv3} and \eqref{eq:conv4}, 
$$
|\nabla_x \rho^n| \sqrt{ \left( \frac{|m|^2}{\rho}\right)^n  } \to |\nabla_x \rho|  \sqrt{ \frac{ |m|^2}{\rho} }  \quad \mbox{in} \;  L^1([0,T] \times \rT^d) \, .
$$
This completes the proof of \eqref{pr:erel2} and the lemma.
\end{proof}

Now we prove a stability estimate which immediately implies weak-strong uniqueness.
We restrict to cases where one of the two following hypotheses holds:
\begin{align}
\label{a1}\  \gamma &\geq 2 \tag{A1} \\
\label{a2} 1 < \gamma < 2  \quad &\text{and } \quad  \int_{\rT^d} \rho_0 = \int_{\rT^d} \bar \rho_0  \tag{A2}
\end{align}

\begin{theorem}
Let $\kappa (\rho) = C_\kappa$ constant and \eqref{a1}, \eqref{a2} hold. Consider $(\rho, m)$ a dissipative (or  conservative)  weak solution of \eqref{eq:EK} satisfying $(\bf H)$ on $[0,T) \times \rT^d$, $T>0$,
and 
let  $(\brho, \barm)$ be a strong solution  of \eqref{eq:EK} satisfying for some  $\delta >0$ the bound $\brho(t,x) \geq \delta $  on $[0,T) \times \rT^d.$
 Then there exists a constant $C>0$ depending on $(\brho , \bar m)$ and their derivatives  so that  for almost every $t \in [0,T):$
 \begin{multline}\label{eq:wsstab}
\int_{\rT^d} \left. \Big ( \frac{1}{2} \rho \Big | \frac{m}{\rho} -  \frac{\bar m}{\bar\rho}\Big | ^2 +  h_\gamma(\rho  | \, \bar \rho \right. ) + \frac{C_\kappa}{2} | \nabla_x\rho - \nabla_x\brho|^2 \Big )dx \Big |_t\\
\leq
  e^{Ct} \Big( \int_{\rT^d} \left. \left ( \frac{1}{2} \rho \Big | \frac{m}{\rho} -  \frac{\bar m}{\bar\rho}\Big | ^2 +  h_\gamma(\rho  | \, \bar \rho \right. ) + \frac{C_\kappa}{2} | \nabla_x\rho - \nabla_x\brho|^2 \right )dx \Big |_{t = 0} \Big) 
 \end{multline}
\end{theorem}

\begin{proof}
For convenience let us define
\begin{equation}
 \begin{split}
  F_\gamma (\rho, \nabla_x \rho |\brho, \nabla_x  \brho) &:=    h_\gamma(\rho  | \, \bar \rho ) + \frac{C_\kappa}{2} | \nabla_x\rho - \nabla_x\brho|^2 ;\\
  K(\rho,m|\brho,\bar m) &:= \frac{1}{2} \rho \Big | \frac{m}{\rho} -  \frac{\bar m}{\bar\rho}\Big | ^2;\\
  p_e (\rho)&:= \rho e'(\rho) - e(\rho).
 \end{split}
\end{equation}

 By subtracting \eqref{eq:erel} from \eqref{eq:RelEnKorGenFinalweak} we obtain
\begin{align}\label{eq:redrel}
& \int_{\rT^d} F_\gamma (\rho, \nabla_x \rho |\brho, \nabla_x  \brho) +  K(\rho,m|\brho,\bar m)dx \Big |^t_0
\\
 &\leq - \int \!\!\int_{[0,t)\times\rT^d}  
 \rho \left( \frac{m}{\rho}-   \frac{\bar m}{\bar\rho}\right )\otimes \left( \frac{m}{\rho}-   \frac{\bar m}{\bar\rho}\right ) :\nabla_x  \left(    \frac{\bar m}{\bar\rho}\right )
 dsdx
 \\
&\ - \int \!\!\int_{[0,t)\times\rT^d}    \dx \left (\frac{\bar m}{\bar\rho}\right ) \Big (  s(\rho, \nabla_x\rho\left | \bar \rho, \nabla_x\bar \rho \right. ) + p(\rho | \brho) \Big )  dsdx
 \nonumber\\
 &\ 
 - \int \!\!\int_{[0,t)\times\rT^d} \left [ \nabla_x \left (\frac{\bar m}{\bar\rho} \right) : H(\rho, \nabla_x\rho\left | \bar \rho, \nabla_x\bar \rho \right. ) 
 + \nabla_x \dx \left (\frac{\bar m}{\bar\rho}\right ) \cdot r(\rho, \nabla_x\rho\left | \bar \rho, \nabla_x\bar \rho \right. )\right]dsdx \nonumber\\
&\ + \int\!\! \int_{[0,t)\times\rT^d} \div ( \barm) e'(\rho|\brho) -  \big( e''(\rho)\nabla_x \rho - e''(\brho)\nabla_x \brho\big) \big( m - \barm \big) \, ds dx\nonumber
.
\end{align}
In view of \eqref{eq:rsH} some terms on the right hand side of \eqref{eq:redrel} are estimated in a straightforward fashion:
\begin{equation}\label{eq:redrel1}
 \begin{split}
 \Big| \rho \left( \frac{m}{\rho}-   \frac{\bar m}{\bar\rho}\right )\otimes \left( \frac{m}{\rho}-   \frac{\bar m}{\bar\rho}\right ) :\nabla_x  \left(    \frac{\bar m}{\bar\rho}\right )\Big|
 &\leq C K(\rho,m|\brho,\bar m),
 \\
 \nabla_x \left (\frac{\bar m}{\bar\rho} \right) : H(\rho, \nabla_x\rho\left | \bar \rho, \nabla_x\bar \rho \right. ) &\leq C F_\gamma (\rho, \nabla_x \rho |\brho, \nabla_x  \brho).
 \end{split}
\end{equation}
Using the decomposition \eqref{decomp_h}, equation \eqref{eq:rsH} and the fact that for $\gamma$-laws pressure and inner energy coincide up to a constant we find
\begin{multline}\label{eq:redrel2}
 s(\rho, \nabla_x\rho | \bar \rho, \nabla_x\bar \rho  ) + p(\rho | \brho) 
 = (\gamma -1) h_\gamma(\rho|\brho) 
 + p_e(\rho|\brho) + \frac{C_\kappa}{2} |\nabla_x \rho - \nabla_x \brho|^2
\\
 \leq C F_\gamma (\rho, \nabla_x \rho |\brho, \nabla_x  \brho) + p_e(\rho|\brho).
\end{multline}
Inserting \eqref{eq:redrel1} and \eqref{eq:redrel2} into \eqref{eq:redrel} we get
\begin{align}\label{eq:redrel3}
& \int_{\rT^d} F_\gamma (\rho, \nabla_x \rho |\brho, \nabla_x  \brho)+ K(\rho,m|\brho,\bar m)dx \Big |_0^t \\
 &\le  C \int \!\!\int_{[0,t)\times\rT^d}  F_\gamma (\rho, \nabla_x \rho |\brho, \nabla_x  \brho) + K(\rho,m|\brho,\bar m)ds dx
 \nonumber\\
 &\  - \int \!\!\int_{[0,t)\times\rT^d} \left [
  \dx \left (\frac{\bar m}{\bar\rho}\right ) p_e(\rho \left | \bar \rho \right. )
 + \nabla_x \dx \left (\frac{\bar m}{\bar\rho}\right ) \cdot r(\rho, \nabla_x\rho\left | \bar \rho, \nabla_x\bar \rho \right. )\right]dsdx \nonumber\\
&\ + \int \!\!\int_{[0,t)\times\rT^d} \div ( \barm)  e'(\rho|\brho) - \big( e''(\rho)\nabla_x \rho - e''(\brho)\nabla_x \brho\big) \big( m - \barm \big) \, ds dx\nonumber
.
\end{align}

For the remaining terms the $L^2$ norm of $(\rho - \brho)$ has to be estimated. In case of \eqref{a1}, in the range $\gamma \ge 2$, 
Lemma 2.4 in \cite{LT13} asserts that there exist constants $R_0,\,C_1,\,C_2$ depending on $\brho$ such that
 \begin{equation*}
  h_\gamma (\rho | \rho) \geq \left\{
  \begin{array}{cl}
  C_1 |\rho - \brho|^2 & \text{ for } \rho\leq R_0,\\
   C_2 |\rho - \brho|^\gamma & \text{ for } \rho> R_0.
  \end{array}
  \right.
 \end{equation*}
In this estimate, exploiting the fact that $\brho$ is a bounded solution, $R_0$ can be chosen sufficiently  large 
so that for $\gamma \geq 2$ we have 
\begin{equation}
\label{fact1}
|\rho - \brho|^2  \le 
 C  \,  h_\gamma (\rho | \brho)   \quad \forall \rho > 0 \, .
\end{equation}
In case of \eqref{a2}  mass conservation implies that $\rho(t,\cdot) - \bar \rho(t,\cdot)$ has mean value zero for any $t \in [0,T]$ such that
\begin{equation}
\label{fact2}
 \| \rho - \bar \rho\|_{L^2(\rT^d)}  \leq C \| \nabla_x \rho - \nabla_x \bar \rho\|_{L^2(\rT^d)} .
\end{equation}

Using now \eqref{fact1} or \eqref{fact2} we estimate the remaining terms. For  
$$
r(\rho, \nabla_x\rho\left | \bar \rho, \nabla_x\bar \rho \right. )=C_\kappa (\rho - \bar \rho) (\nabla_x \rho - \nabla_x \bar \rho)
$$  
we obtain, using Young's inequality,
\begin{equation}
 \int_{\rT^d} \left| \nabla_x \dx \left (\frac{\bar m}{\bar\rho}\right ) \cdot r(\rho, \nabla_x\rho\left | \bar \rho, \nabla_x\bar \rho \right. ) \right| \, dx \leq C  \int_{\rT^d} F_\gamma (\rho, \nabla_x \rho |\brho, \nabla_x  \brho)\, dx.
\end{equation}
We also obtain, since all derivatives of $e$ and $p_e$ are uniformly bounded,
\begin{equation}
 \begin{split}
 \int_{\rT^d} \left|  \div ( \barm)  e'(\rho|\brho) \right|\, dx & \leq C  \int_{\rT^d} F_\gamma (\rho, \nabla_x \rho |\brho, \nabla_x  \brho) \, dx,\\
 \int_{\rT^d} \left|\dx \left (\frac{\bar m}{\bar\rho}\right ) p_e(\rho\left | \bar \rho \right. ) \right|\, dx & \leq C \int_{\rT^d} F_\gamma (\rho, \nabla_x \rho |\brho, \nabla_x  \brho) \, dx.
 \end{split}
\end{equation}
Inserting these estimates into \eqref{eq:redrel3} we obtain
\begin{align}
\label{eq:redrel4}
& \int_{\rT^d} F_\gamma (\rho, \nabla_x \rho |\brho, \nabla_x  \brho)+ K(\rho,m|\brho,\bar m)dx \Big |_0^t
\\
&\, \le  C \int \!\!\int_{[0,t)\times\rT^d}  F_\gamma (\rho, \nabla_x \rho |\brho, \nabla_x  \brho)+ K(\rho,m|\brho,\bar m)ds dx
 \nonumber
 \\
& \quad -
   \int \!\!\! \int_{[0,t)\times\rT^d} \big( e''(\rho)\nabla_x \rho - e''(\brho)\nabla_x \brho\big) \big( m - \barm \big)\, ds dx\nonumber .
\end{align}
We rewrite
\begin{multline}\label{eq:redrel5}
 \big( e''(\rho)\nabla_x \rho - e''(\brho)\nabla_x \brho\big) \big( m - \barm \big)\\
 = \Big( e''(\rho) \big(\nabla_x \rho - \nabla_x \brho\big) + \big(e''(\rho) - e''(\brho)\big) \nabla_x \brho \Big) \Big( \rho (u - \bar u) + \bar u (\rho - \brho) \Big).
\end{multline}
Since $\sqrt{\rho} e''(\rho)$ is bounded uniformly in $\rho$ we have
\begin{equation}\label{eq:redrel6}
 \begin{split}
  \Big| \sqrt{\rho} e''(\rho) \big(\nabla_x \rho - \nabla_x \brho\big)  \sqrt{\rho}  (u - \bar u)  \Big| & \leq C F_\gamma (\rho, \nabla_x \rho |\brho, \nabla_x  \brho) + C  K(\rho,m|\brho,\bar m);\\
  \Big| e''(\rho) \big(\nabla_x \rho - \nabla_x \brho\big)\bar u (\rho - \brho) \Big| & \leq C F_\gamma (\rho, \nabla_x \rho |\brho, \nabla_x  \brho); \\
  \Big| \big(e''(\rho) - e''(\brho)\big) (\nabla_x \brho) \bar u (\rho - \brho)\Big| & \leq C F_\gamma (\rho, \nabla_x \rho |\brho, \nabla_x  \brho).
 \end{split}
\end{equation}
It remains to derive a bound for 
\[ \big(e''(\rho) - e''(\brho)\big) \nabla_x \brho \, \rho (u - \bar u).\]
It suffices to show that $|\big(e''(\rho) - e''(\brho)\big) \sqrt{\rho}|$ is uniformly bounded by $C | \rho - \brho|.$
This is trivial, as long as $\rho$ is bounded from above so we restrict ourselves to the case $\rho >R+1 $ where $R$ is a constant satisfying
\[ R > \sup_{t,x} \brho(t,x)  \text{ and }  e(r) =0 \quad \forall \ r> R  .\]
In that case
\begin{multline}\label{eq:redrel7}
 \Big |\big(e''(\rho) - e''(\brho)\big) \sqrt{\rho}\Big | =\Big | \big(e''(\rho) - e''(\brho)\big) \sqrt{R} + \big(e''(\rho) - e''(\brho)\big)(\sqrt{\rho} - \sqrt{R})\Big |\\
  = \Big |\big(e''(\rho) - e''(\brho)\big) \sqrt{R} + \big(- e''(\brho)\big)\frac{\rho - R}{\sqrt{\rho} + \sqrt{R}}\Big |
  \leq C | \rho - \brho| \sqrt{R} + \frac{C}{2 \sqrt{R}} | \rho - \brho|.
\end{multline}
By combining \eqref{eq:redrel6} and \eqref{eq:redrel7} and inserting them into \eqref{eq:redrel4} we obtain
\begin{multline}\label{eq:redrel8}
 \int_{\rT^d} F_\gamma (\rho, \nabla_x \rho |\brho, \nabla_x  \brho) +  K(\rho,m|\brho,\bar m)dx \Big |_t
\leq
 \int_{\rT^d} F_\gamma (\rho, \nabla_x \rho |\brho, \nabla_x  \brho)+ K(\rho,m|\brho,\bar m)dx  \Big |_0 \\
 + C \int \!\!\int_{[0,t)\times\rT^d}  F_\gamma (\rho, \nabla_x \rho |\brho, \nabla_x  \brho)+ K(\rho,m|\brho,\bar m)\, ds dx.
\end{multline}
The theorem follows from applying Gronwall's lemma to \eqref{eq:redrel8}.
\end{proof}


\section{From Euler-Korteweg with large friction to Cahn-Hilliard}\label{sec:lf}
In this section we study the relaxation limit from the Euler-Korteweg system with large friction to the Cahn-Hilliard equation. This limit was studied for monotone (increasing) pressures, i.e., convex energies, in \cite{LT16}
and we will show that the same techniques as in Section \ref{sec:wsnce} allow us to extend the results to the non-monotone case. As before we will assume that the energy density has a decomposition \eqref{decomp_h} with \eqref{a1} or \eqref{a2} and $\kappa(\rho) \equiv C_\kappa > 0$.
After rescaling time the Euler-Korteweg system with friction can be written as (cf. \cite[Sec. 4]{LT16})
\begin{equation}\label{eq:eklf}
\begin{split}
\rho_t + \frac{1}{\veps} \div m &=0\\
m_t + \frac{1}{\veps} \div \big( \frac{m \otimes m}{\rho} \big) &=
- \frac{1}{\veps^2} m - \frac{1}{\veps}
\rho \nabla_x \big( h'(\rho) - C_\kappa \Delta_x \rho\big).
\end{split}
\end{equation}
Our goal is to investigate the limit $\veps \rightarrow 0$ of \eqref{eq:eklf}.
Formally the limit equation is
\begin{equation}\label{eq:CH}
\rho_t - \div \big( \rho \nabla_x \big( h'(\rho) - C_\kappa \Delta_x \rho \big) \big)=0.
\end{equation}
We will make this rigorous on time intervals for which \eqref{eq:CH} possesses a strong solution and \eqref{eq:eklf} has weak solutions for all (sufficiently small) $\veps>0$.

To this end we will establish a relative energy framework comparing weak solutions to 
\eqref{eq:eklf} and classical solutions to \eqref{eq:CH}.

\begin{definition} \label{def:wksolf}
(i) A function $( \rho,  m)$ with  $\rho \in C([0, \infty) ; L^1 ( \rT^d ) )$, $m \in C  \big (  [0, \infty) ;   L^1(\rT^d, \rR^d)  \big )$,
$\rho \ge 0$,  is a weak
solution of \eqref{eq:eklf}, if $\frac{m \otimes m}{\rho}$, $S \in L^1_{loc}  \left (  (0, \infty) \times \rT^d ) \right )^{d \times d}$,  and $(\rho, m)$ satisfy
\begin{equation}\label{eq:wksolf}
\begin{aligned}
- \iint \rho \psi_t + \frac{1}{\veps}m \cdot \nabla_x \psi dx d t  = \int \rho ( 0,x) \psi (0,x) dx \, ,  \qquad  \forall \psi \in C^1_c \left ( [0, \infty) ; C^1 (\rT^d) \right )\, ;
\\
- \iint m \cdot \varphi_t + \frac{1}{\veps}\frac{m \otimes m}{\rho} : \nabla_x \varphi - \frac{1}{\veps}S : \nabla_x \varphi - \frac{1}{\veps^2} m \varphi \,  dx dt = \int m(0,x) \cdot \varphi(0,x) dx  \, ,  \qquad \qquad 
\\
 \forall \varphi  \in C^1_c \left ( [0, \infty) ;  \big ( C^1 (\rT^d)  \big )^d \right ) \, .
\end{aligned}
\end{equation}

\medskip\noindent
(ii) If,  in addition, $\frac{1}{2} \frac{|m|^2}{\rho} + h(\rho) + \frac{\kappa(\rho)}{2} |\nabla_x \rho|^2 \in C([0, \infty) ; L^1 ( \rT^d ) )$ and it satisfies
\begin{multline}
 \label{eq:diswsf}
  - \iint \left ( \frac{1}{2} \frac{|m |^2}{\rho} + h(\rho) + \frac{\kappa(\rho)}{2} |\nabla_x \rho|^2  \right ) \dot\theta(t) \, dx dt\\
 \le   \int \left ( \frac{1}{2} \frac{|m |^2}{\rho } + h(\rho) + \frac{\kappa(\rho)}{2} |\nabla_x \rho|^2  \right ) \Big |_{t=0}  \theta(0)dx 
 - \frac{1}{\veps^2} \iint \frac{|m|^2}{\rho} \theta(t) dt dx,
 \end{multline}
 for any non-negative $ \theta  \in W^{1, \infty} [0, \infty)$ compactly supported on $[0, \infty)$
then $(\rho, m)$ is called a dissipative weak solution. 

 \medskip\noindent
(iii) By contrast,
if $\frac{1}{2} \frac{|m|^2}{\rho} +h(\rho) + \frac{\kappa(\rho)}{2} |\nabla_x \rho|^2  \in C([0, \infty) ; L^1 ( \rT^d ) )$ and it satisfies \eqref{eq:diswsf}
as an equality, then $(\rho, m)$  is called a conservative weak solution.
\end{definition}

Note that solutions to \eqref{eq:CH} can be understood as solutions to \eqref{eq:eklf}
with forcing. 
For any solution $\brho $ of \eqref{eq:CH} we set
\begin{equation}
\label{order1}
 \bar m = - \veps \brho (h'(\brho) - C_\kappa \Delta_x \brho) = \mathcal{O}(\veps)
\end{equation}
 and 
\begin{equation}
\label{order2}
\bar E:= \veps \div \Big( \brho \nabla_x (h'(\brho) - C_\kappa \Delta_x \brho)
\otimes \nabla_x (h'(\brho) - C_\kappa \Delta_x \brho)\Big) -
\veps \big( \brho (h'(\brho) - C_\kappa \Delta_x \brho ) \big)_t= \mathcal{O}(\veps).
\end{equation}
Then, $(\brho, \bar m)$ satisfy
\begin{equation}\label{eq:ekch}
\begin{split}
\brho_t + \frac{1}{\veps} \div \bar m &=0\\
\bar m_t + \frac{1}{\veps} \div \big( \frac{\bar m \otimes\bar m}{\brho} \big) &=
- \frac{1}{\veps^2} \bar m - \frac{1}{\veps}
\brho \nabla_x \big( h'(\brho) - C_\kappa \Delta_x \brho\big) + \bar E \, .
\end{split}
\end{equation}
For smooth $\brho$ the forcing term is of order $O(\eps)$ and is visualized henceforth as an error term.

\begin{definition}\label{def:sch}
 We call $\brho$ a strong solution of \eqref{eq:CH} on $[0,T) \times \rT^d$ provided
 \begin{align}
  \rho \in C^0([0,T), C^4(\rT^d)) \cap C^1([0,T), C^2(\rT^d))
 \end{align}
and \eqref{eq:CH} is satisfied in a point-wise sense.
\end{definition}
Note that the regularity required in definition \ref{def:sch} is more than what we need to give a point-wise meaning to each term in \eqref{eq:CH}.
The imposed regularity makes $(\brho,\bar m)$ a strong solution of \eqref{eq:ekch}.

Next, \cite[Thm 4.2]{LT16} implies for any weak (dissipative or conservative) weak solution $(\rho,m)$ of \eqref{eq:eklf} and any strong solution $\brho$ of \eqref{eq:CH} the following inequality 
\begin{equation}
 \label{eq:RelEnKorlf}
\begin{split}
& \int_{\rT^d} \left. \Big ( \frac{1}{2} \rho \Big | \frac{m}{\rho} -  \frac{\bar m}{\bar\rho}\Big | ^2 +  h(\rho  | \, \bar \rho \right. )+ \frac{C_\kappa}{2}|\nabla_x \rho - \nabla_x \brho|^2 \Big )dx \Big |_t\\
&\leq
 \int_{\rT^d} \left. \left ( \frac{1}{2} \rho \Big | \frac{m}{\rho} -  \frac{\bar m}{\bar\rho}\Big | ^2 +  h(\rho  | \, \bar \rho \right. )+\frac{C_\kappa}{2}|\nabla_x \rho - \nabla_x \brho|^2  \right )dx \Big |_0
\\
&-\frac{1}{\veps^2} \int \!\!\int_{[0,t)\times\rT^d}\rho \Big | \frac{m}{\rho} -  \frac{\bar m}{\bar\rho}\Big | ^2\, ds dx
- \int \!\!\int_{[0,t)\times\rT^d} \bar E \frac{\rho}{\bar \rho} \Big( 
\frac{m}{\rho} - \frac{\bar m}{\bar \rho}
\Big)\, ds dx
\\ 
 &\  -\frac{1}{\veps} \int \!\!\int_{[0,t)\times\rT^d} \left [\rho \left( \frac{m}{\rho}- \frac{\bar m}{\bar\rho}\right )\otimes \left( \frac{m}{\rho}-  \frac{\bar m}{\bar\rho}\right ) :\nabla_x  \left( \frac{\bar m}{\bar\rho}\right )\right] \, ds dx
 \\
  &\  -\frac{1}{\veps} \int \!\!\int_{[0,t)\times\rT^d} 
  \left[\dx \left (\frac{\bar m}{\bar\rho}\right ) \Big(\frac{C_\kappa}{2}|\nabla_x \rho - \nabla_x \brho|^2   + p(\rho|\bar \rho)   \Big)  \right]dsdx
\\
 &\ 
 -\frac{C_\kappa}{\veps} \int \!\!\int_{[0,t)\times\rT^d} \left [ \nabla_x \left (\frac{\bar m}{\bar\rho} \right) : \nabla_x(\rho - \brho) \otimes \nabla_x(\rho - \bar \rho) 
 + \nabla_x \dx \left (\frac{\bar m}{\bar\rho}\right )(\rho- \brho) \nabla_x(\rho -\bar \rho) \right]dsdx\, .
\end{split}
\end{equation}

Invoking Lemma \ref{lem:erel} we obtain the following estimate from \eqref{eq:RelEnKorlf}
\begin{equation}
 \label{eq:RelEnKorlf2}
\begin{split}
& \int_{\rT^d} \left. \Big ( \frac{1}{2} \rho \Big | \frac{m}{\rho} -  \frac{\bar m}{\bar\rho}\Big | ^2 +  h_\gamma(\rho  | \, \bar \rho \right. )+ \frac{C_\kappa}{2}|\nabla_x \rho - \nabla_x \brho|^2 \Big )dx \Big |_t\\
&\leq
 \int_{\rT^d} \left. \left ( \frac{1}{2} \rho \Big | \frac{m}{\rho} -  \frac{\bar m}{\bar\rho}\Big | ^2 +  h_\gamma(\rho  | \, \bar \rho \right. )+\frac{C_\kappa}{2}|\nabla_x \rho - \nabla_x \brho|^2  \right )dx \Big |_0
\\
&-\frac{1}{\veps^2} \int \!\!\int_{[0,t)\times\rT^d}\rho \Big | \frac{m}{\rho} -  \frac{\bar m}{\bar\rho}\Big | ^2\, ds dx
- \int \!\!\int_{[0,t)\times\rT^d} \bar E \frac{\rho}{\bar \rho} \Big( 
\frac{m}{\rho} - \frac{\bar m}{\bar \rho}
\Big)\, ds dx
\\ 
 &\  -\frac{1}{\veps} \int \!\!\int_{[0,t)\times\rT^d} \left [\rho \left( \frac{m}{\rho}-   \frac{\bar m}{\bar\rho}\right )\otimes \left( \frac{m}{\rho}-   \frac{\bar m}{\bar\rho}\right ) :\nabla_x  \left( \frac{\bar m}{\bar\rho}\right )\right] \, ds dx
 \\
  &\  -\frac{1}{\veps} \int \!\!\int_{[0,t)\times\rT^d} 
  \left[\dx \left (\frac{\bar m}{\bar\rho}\right ) \Big(\frac{C_\kappa}{2}|\nabla_x \rho - \nabla_x \brho|^2   + p(\rho|\bar \rho)   \Big)  \right]dsdx
\\
 &\ 
 -\frac{C_\kappa}{\veps} \int \!\!\int_{[0,t)\times\rT^d} \left [ \nabla_x \left (\frac{\bar m}{\bar\rho} \right) : \nabla_x(\rho - \brho) \otimes \nabla_x(\rho - \bar \rho) 
 + \nabla_x \dx \left (\frac{\bar m}{\bar\rho}\right )(\rho- \brho) \nabla_x(\rho -\bar \rho) \right]dsdx\\
 &\ +\frac{1}{\veps} \int\!\! \int_{[0,t)\times\rT^d} \div ( \barm)  e'(\rho|\brho)  - \big( e''(\rho)\nabla_x \rho - e''(\brho)\nabla_x \brho\big) \big( m - \barm \big)\, ds dx \, .
\end{split}
\end{equation}

Following the arguments in the proof of Theorem 3.3 and using the orders in \eqref{order1} and \eqref{order2}, 
we obtain:

\begin{theorem}
 Let $(\brho, \barm)$ be a strong solution of \eqref{eq:CH} on $[0,T) \times \rT^d$ for some $T>0,$ for which there exists $\delta >0$ such that $\brho(t,x) \geq \delta $ for all $(t,x) \in [0,T) \times \rT^d.$
 Let \eqref{a1} or \eqref{a2} hold.
 Then for any conservative or dissipative weak solution
 $(\rho, m)$   of  \eqref{eq:eklf} on $[0,T) \times \rT^d$ satisfying $(\bf H)$ the function
\begin{equation}
\Psi_\eps (t):=  \int_{\rT^d} \left. \Big ( \frac{1}{2} \rho \Big | \frac{m}{\rho} -  \frac{\bar m}{\bar\rho}\Big | ^2 +  h_\gamma(\rho  | \, \bar \rho \right. ) + \frac{C_\kappa}{2} | \nabla_x\rho - \nabla_x\brho|^2 \Big )dx \Big |_t
\end{equation}
 satisfies the following estimate for almost every $t \in [0,T):$
 \[ \Psi_\eps (t) \leq e^{Ct}( \Psi_\eps  (0) + \veps^4)\]
 with $C$ a positive constant depending only on $T$, $K_1$, $\bar \rho$, $\bar m$ and their derivatives. Moreover, if $\Psi_\eps (0) \rightarrow 0$ 
 as $\veps \rightarrow 0$, then
 \[ \sup_{t \in [0,T]} \Psi_\eps (t) \rightarrow 0 \quad \text{as} \; \veps \rightarrow 0.\]
\end{theorem}

\section{Vanishing Capillarity Limit}\label{sec:vcl}
In this section we study the vanishing capillarity limit in the case of  convex energy densities $h.$
For convenience we use a $\gamma$-law $h(\rho) = \rho^\gamma$ with $\gamma>1.$
We consider two different settings for the capillarity. The first setting is constant capillarity
\begin{equation}\label{set1}\tag{Set1}
 \kappa(\rho)=C_\kappa>0.
\end{equation}
The second setting is:
\begin{equation}\label{set2}\tag{Set2}
\begin{split}
 \kappa(\rho) > 0, \ \kappa(\rho) \kappa''(\rho) - 2(\kappa'(\rho))^2\geq 0, \quad 
  \rho^2 \kappa(\rho) \lesssim h(\rho) + \rho, \quad |\rho \kappa'(\rho)| \lesssim \kappa(\rho) ,
\end{split}
\end{equation}
for all $\rho > 0$. Recall that the notation  $a \lesssim b$ for two positive quantities $a,b$ indicates that there is a constant $C>0$ such that $a \leq C b.$

\begin{remark}[Setting 2]
\begin{enumerate}
 \item  The conditions on $\kappa $ in \eqref{set2} ensure  that the Hessian of 
 $F(\rho,q) := \tfrac{\kappa(\rho)}{2} |q|^2$ given by
\[\nabla_{(\rho,q)}^2 F (\rho,q) = \begin{pmatrix}
 \frac12 \kappa''(\rho)|q|^2 & \kappa'(\rho)q \\
\kappa'(\rho)q^T & \kappa(\rho)\mathbb{I}
\end{pmatrix}\]
is positive semi-definite. This, in particular, implies
\[ F(\rho,q|\brho , \bar q) \geq 0 \quad \forall \ \rho,\brho \geq 0,\, q , \bar q \in \mathbb{R}^d.\]
\item The assumptions of setting 2 cover, in particular, the quantum hydrodynamics case $\kappa(\rho)=\rho^{-1}.$
\end{enumerate}
\end{remark}

We will fix from now on some $\kappa$ satisfying \eqref{set1} or \eqref{set2}
 and investigate the limit for $\varepsilon \rightarrow 0$ of dissipative or conservative weak solutions $(\rho^\veps,\rho^\veps u^\veps)$ of
\begin{equation} \label{eq:EK-eps}
    \begin{split}
	\rho^\veps_{t} +\dx (\rho^\veps u^\veps) &=0  \\
    (\rho^\veps u^\veps)_{t} + \dx (\rho^\veps u^\veps \otimes u^\veps)
	&=  - \rho \nabla_x \Big( h'(\rho^\veps) + \frac{\veps \kappa'(\rho^\veps)}{2} |\nabla_x \rho^\veps|^2  -\dx (\veps \kappa(\rho^\veps) \nabla_x \rho^\veps)  \Big )\, .  \\
    \end{split}
\end{equation}

We will show that $(\rho^\veps,\rho^\veps u^\veps)$ converges, for $\varepsilon \rightarrow 0$, to a solution $(\rho,\rho u)$ of 
\begin{equation} \label{eq:Euler}
    \begin{split}
	\rho_{t} +\dx (\rho u) &=0  \\
    (\rho u)_{t} + \dx (\rho u \otimes u)
	&=  - \rho \nabla_x \Big( h'(\rho)  \Big )\, ,  \\
    \end{split}
\end{equation}
on any time interval $[0,T]$ such that \eqref{eq:Euler} admits a solution  satisfying
\begin{equation}\label{Euler-reg}
  \begin{split}
  \rho &\in C^0([0,T],C^3(\rT^d,\mathbb{R}_+)) \cap C^{1}((0,T),C^1(\rT^d,\mathbb{R}_+))\, ,\\
   u &\in C^0([0,T], C^2(\rT^d,\mathbb{R}^d)) \cap C^1((0,T),C^0(\rT^d,\mathbb{R}^d))\, .
 \end{split}
\end{equation}

\begin{remark}[Regularity]
Note that it is not sufficient for  $(\rho,\rho u)$ to be a classical solution of \eqref{eq:Euler} but it needs to have the regularity of classical solutions to \eqref{eq:EK}
and, in addition, second (spatial) derivatives of the velocity need to exist.
\end{remark}

The convergence results we obtain in both settings are similar, and we use relative energy in both cases. However, there is a difference
in strategy in the two proofs. In setting 1, we use the Euler-Korteweg relative energy, while in setting 2 we use the
relative energy for the limiting Euler system adapted to account for the Euler-Korteweg energy of $(\rho^\veps, u^\veps)$.

\begin{theorem}\label{thrm:set1}
Let $\kappa$ satisfying \eqref{set1} be given.
 Let $(\rho,u)$ be a solution of \eqref{eq:Euler} with initial data $(\rho_0,u_0)$ satisfying \eqref{Euler-reg}.
 Let $\{(\rho^\veps,u^\veps)\}_{\veps>0}$ be a family of (conservative or dissipative) weak solutions to \eqref{eq:EK-eps} parametrized in $\veps$ 
 having the same initial data $(\rho_0,u_0).$
 Then, there exists a constant $C>0$ depending on $T$ and  $(\rho,u)$ such that 
 \begin{equation}
  \int_{\rT^d} h(\rho^\veps(t,\cdot)|\rho(t,\cdot)) + \frac{\rho^\veps(t,\cdot)}{2} \big|u^\veps (t,\cdot)- u(t,\cdot)\big|^2 \, dx \leq  C \veps^2 
 \end{equation}
for almost all $t \in [0,T].$
\end{theorem}

\begin{proof}
 Since the initial data coincide and $\kappa (\rho) =C_\kappa$ equation \eqref{eq:RelEnKorGenFinalweak} reads, 
\begin{equation}\label{Econv}
\begin{split}
& \int_{\rT^d} \left. \Big ( \frac{1}{2} \rho^\veps \Big | \frac{m^\veps}{\rho^\veps} -  \frac{m}{\rho}\Big | ^2 +  h(\rho^\veps  | \,  \rho \right. )+ \veps \frac{C_\kappa}{2} | \nabla_x \rho^\veps - \nabla_x \rho|^2\Big )dx \Big |_t\\
&\leq
   - \int \!\!\int_{[0,t)\times\rT^d} \left [\rho^\veps \left( \frac{m^\veps}{\rho^\veps}-   \frac{ m}{\rho}\right )\otimes \left( \frac{m^\veps}{\rho^\veps}-   \frac{ m}{\rho}\right ) :\nabla_x  \left(    \frac{m}{\rho}\right )\right]dsdx\\
 &\  - \int \!\!\int_{[0,t)\times\rT^d}\left[ \dx \left (\frac{m}{\rho}\right ) \Big(\veps s(\rho^\veps, \nabla_x\rho^\veps\left | \rho, \nabla_x \rho \right. ) + p(\rho^\veps|\rho) \Big)\right]dsdx
\\
 &\ 
 - \int \!\!\int_{[0,t)\times\rT^d} \left [ \nabla_x \left (\frac{ m}{\rho} \right) : \veps H(\rho^\veps, \nabla_x\rho^\veps\left | \rho, \nabla_x \rho \right. ) 
 + \nabla_x \dx \left (\frac{m}{\rho}\right ) \cdot \veps r(\rho^\veps, \nabla_x\rho^\veps\left | \rho, \nabla_x\rho \right. )\right]dsdx
 \\
&\   +  \int \!\!\int_{[0,t)\times\rT^d} \varepsilon C_\kappa  \nabla \Delta \rho \cdot   \rho^\varepsilon \left (  \frac{m^\eps}{\rho^\eps}  - \frac{m}{\rho}  \right ) \, ds dx
\end{split}
\end{equation}
with
\begin{align*}
 H(\rho^\veps, \nabla_x^\veps\rho\left | \rho, \nabla_x\rho \right. )  
 &= C_\kappa(\nabla_x\rho^\veps-\nabla_x\rho)\otimes(\nabla_x\rho^\veps-\nabla_x\rho)\, ;\\
  s(\rho^\veps, \nabla_x\rho^\veps\left | \rho, \nabla_x \rho \right. ) 
   &= 
   C_\kappa|\nabla_x\rho^\veps - \nabla_x\rho|^2 
   \,
     ;\\
  r(\rho^\veps, \nabla_x\rho^\veps\left |  \rho, \nabla_x\rho \right. )  &= C_\kappa \big(\rho^\veps - \rho \big) (\nabla_x \rho^\veps -\nabla_x\rho)\, ;
  \end{align*}
  confer \cite[Sec 3.1]{GLT}.
Thus, there exists a constant $C>0$ such that
\begin{equation}
 \begin{split}
\big| H(\rho^\veps, \nabla_x^\veps\rho\left | \rho, \nabla_x\rho \right. )   \big| &\leq  C| \nabla_x \rho^\veps - \nabla_x \rho|^2 ;\\
\big| s(\rho^\veps, \nabla_x^\veps\rho\left | \rho, \nabla_x\rho \right. )   \big| &\leq C | \nabla_x \rho^\veps - \nabla_x \rho|^2 ;\\
\big| r(\rho^\veps, \nabla_x^\veps\rho\left | \rho, \nabla_x\rho \right. )   \big| &\leq C  | \nabla_x \rho^\veps - \nabla_x \rho|^2;
 \end{split}
\end{equation}
where we have used Poincar{\'e}'s  inequality in the third estimate.

The last term in \eqref{Econv} is estimated by
$$
\Big | \int \!\!\int_{[0,t)\times\rT^d} \varepsilon C_\kappa  \nabla \Delta \rho \cdot   \rho^\varepsilon \left (  \frac{m^\eps}{\rho^\eps}  - \frac{m}{\rho}  \right ) \, ds dx \Big |
\le  \int \!\!\int_{[0,t)\times\rT^d}   \frac{1}{2} \rho^\eps | u^\eps - u|^2 ds dx + \eps^2 C t \int_{\rT^d} \rho^\eps dx
$$
where $\int \rho^\eps dx \le C$ is uniformly bounded by the conservation of mass. 
Moreover, $p(\rho)=(\gamma -1 ) h(\rho)$ and thus the function
\begin{equation}
 \Phi_\veps(t):= \int_{\rT^d} \left. \Big ( \frac{1}{2} \rho^\veps \Big | \frac{m^\veps}{\rho^\veps} -  \frac{m}{\rho}\Big | ^2 +  h(\rho^\veps  | \,  \rho \right. )+ \veps \frac{C_\kappa}{2}| \nabla_x \rho^\veps - \nabla_x \rho|^2 \Big )dx \Big |_t
\end{equation}
satisfies 
\[ 
\Phi_\veps(t) \leq \int_0^t C \Phi_\veps(s) \, ds +  \veps^2  C t \, , \quad \Phi_\veps (0)=0  
\]
for some constant $C>0$ independent of $\veps.$
Hence,
\[ 
\Phi_\veps(t) \leq \veps^2 C \exp( C t)
\]
which completes the proof of the theorem.
\end{proof}

\begin{remark}[Initial data]
It is straightforward to see that a result analogous to Theorem \ref{thrm:set1} holds 
in case $(\rho^\veps,u^\veps)$ has initial data $(\rho_0^\veps, u^\veps_0)$ such that
 \begin{equation}
 \| \rho_0^\veps\|_{H^1(\rT^d)} = \mathcal{O}(1),\ \int_{\rT^d} h(\rho^\veps_0|\rho_0) + \rho^\veps_0 \big|u^\veps_0- u_0\big|^2 \, dx = \mathcal{O}(\veps^2)\quad
 \text{ 
 and
}\quad
 \int_{\rT^d} \rho^\veps_0 - \rho \, dx =0.
 \end{equation}
 \end{remark}

\begin{theorem}\label{thrm:set2}
Let $\kappa$ satisfying \eqref{set2} be given.
 Let $(\rho,u)$ be a solution of \eqref{eq:Euler} with initial data $(\rho_0,u_0)$ satisfying \eqref{Euler-reg}.
 Let $\{(\rho^\veps,u^\veps)\}_{\veps>0}$ be a family of (conservative or dissipative) weak solutions to \eqref{eq:EK-eps} parametrized in $\veps$ 
 having the same initial data $(\rho_0,u_0).$
 Then, there exists a constant $C>0$ depending on $T$ and  $(\rho,u)$ such that 
 \begin{equation}
  \int_{\rT^d} h(\rho^\veps(t,\cdot)|\rho(t,\cdot)) + \frac{\rho^\veps(t,\cdot)}{2} \big|u^\veps (t,\cdot)- u(t,\cdot)\big|^2 \, dx \leq C \veps
 \end{equation}
for almost all $t \in [0,T].$
\end{theorem}

\begin{proof}
 For any pair of density $\rho$ and momentum $m$ we denote
 \begin{equation}
 \begin{split}
  \eta_\veps (\rho,m)&:= h(\rho) + \frac{1}{2} \frac{|m|^2}{\rho} + \veps \frac{\kappa(\rho)}{2} |\nabla_x \rho|^2;\\
  f(\rho,m)&:= \begin{pmatrix}
                     m \\ \frac{m \otimes m }{\rho}
                     +p(\rho)  
                    \end{pmatrix};
                    \\
 -S_\veps[\rho]&:=\veps \Big[ \frac{\rho \kappa'(\rho) + \kappa(\rho)}{2} |\nabla_x \rho|^2 - \dx ( \rho \kappa(\rho) \nabla_x \rho) \Big ]  \mathbb{I} + \veps \kappa(\rho)\nabla_x \rho \otimes \nabla_x \rho.
  \end{split}
 \end{equation}

Using this notation we can rewrite \eqref{eq:EK-eps}
as
\begin{equation}
 \del_t \begin{pmatrix}
         \rho^\veps \\m^\veps 
        \end{pmatrix} + \dx f(\rho^\eps,m^\eps) - \dx \begin{pmatrix} 0 \\ S_\veps[\rho^\veps]\end{pmatrix} = \begin{pmatrix}
        0\\0
        \end{pmatrix}
\end{equation}
and \eqref{eq:Euler} as
\begin{equation}\label{rew:Euler}
 \del_t \begin{pmatrix}
         \rho \\m
        \end{pmatrix} + \dx f(\rho,m)  = \begin{pmatrix}
        0\\0
        \end{pmatrix}.
\end{equation}
We will monitor the temporal evolution of
\begin{equation}
\begin{split}
 \Phi_\veps(t)&:= \int_{\rT^d} \eta_\veps(\rho^\veps, m^\veps) - \eta_0(\rho,m) - \operatorname{D} \eta_0 (\rho,m) 
 \begin{pmatrix}
          \rho^\veps -\rho \\ m^\veps   - m                                                                                                           
   \end{pmatrix}
   \, dx\Big|_t\\
   &=\int_{\rT^d} \left. \Big ( \frac{1}{2} \rho^\veps \Big | \frac{m^\veps}{\rho^\veps} -  \frac{m}{\rho}\Big | ^2 +  h(\rho^\veps  | \,  \rho \right. )+ 
   \veps \frac{\kappa (\rho^\eps) }{2}   | \nabla_x \rho^\veps |^2 \Big )dx \Big |_t.
   \end{split}
\end{equation}
Note that
\[ \operatorname{D} \eta_0 (\rho,m)  = \begin{pmatrix} 
                                        h'(\rho) - \frac{1}{2} \frac{|m|^2}{\rho^2}\\ \frac{m}{\rho}
                                       \end{pmatrix}.\]
A straightforward computation shows
\begin{multline}\label{phiest0}
 \Phi_\veps(t) \leq   \iint_{[0,t) \times \rT^d } -\Big( \operatorname{D} \eta_0 (\rho,m)\Big)_t \begin{pmatrix}
      \rho^\veps - \rho\\ m^\veps - m \end{pmatrix}\\
   - 
  \nabla_x \Big( \operatorname{D} \eta_0 (\rho,m)\Big)
   \Big( f(\rho^\veps,m^\veps ) - \begin{pmatrix}
              0 \\ S_\veps[\rho^\veps]
    \end{pmatrix}- f(\rho,m) \Big)ds dx.
\end{multline}
Since $(\rho,u)$ is a strong solution of \eqref{rew:Euler} we infer
\begin{multline}\label{phiest1}
 \Phi_\veps(t) \leq \iint_{[0,t)\times \rT^d} \nabla_x \Big( \operatorname{D} \eta_0 (\rho,m)\Big): \Big[ f(\rho^\veps,m^\veps) - f(\rho,m) - \operatorname{D} f(\rho,m)  \begin{pmatrix}
          \rho^\veps - \rho \\ m^\veps    -m                                                                                                         
   \end{pmatrix}\Big]
 \, ds dx\\
   + \iint_{[0,t)\times \rT^d} \nabla_x \Big( \operatorname{D} \eta_0 (\rho,m)\Big) :
   \begin{pmatrix} 0 \\ S_\veps[\rho^\veps]\end{pmatrix}\, ds dx.
\end{multline}

Using the definitions of $f$ and $S_\veps,$ equation \eqref{phiest1} is equivalent to
\begin{equation}\label{phiest2}
 \begin{split}
   \Phi_\veps(t) \leq& 
    - \int \!\!\int_{[0,t)\times\rT^d} \left [\rho^\veps \left( \frac{m^\veps}{\rho^\veps}-   \frac{ m}{\rho}\right )\otimes \left( \frac{m^\veps}{\rho^\veps}-   \frac{ m}{\rho}\right ) :\nabla_x  \left(    \frac{m}{\rho}\right )\right]dsdx\\
 &\  - \int \!\!\int_{[0,t)\times\rT^d}\left[ \dx \left (\frac{m}{\rho}\right )  p(\rho^\veps|\rho)\right]dsdx
\\
&\ - \iint_{[0,t)\times\rT^d} \dx\left (\frac{m}{\rho}\right )\veps \Big[  \frac{\rho^\veps \kappa'(\rho^\veps)+ \kappa(\rho^\veps)}{2}|\nabla_x \rho^\veps|^2 
       - \dx ( \rho^\veps \kappa(\rho^\veps) \nabla_x \rho^\veps) \Big ]\, ds dx\\
&\ - \int \!\!\int_{[0,t)\times\rT^d} \veps \kappa(\rho^\veps) \nabla_x \left (\frac{m}{\rho}\right ): \nabla_x \rho^\veps \otimes \nabla_x \rho^\veps\, ds dx.
 \end{split}
\end{equation}
Due to the hypothesis $|\rho^\veps \kappa'(\rho^\veps)| \lesssim \kappa(\rho^\veps)$ there is a constant $C>0$ such that
\[ \int_{\rT^d} \veps \left|   \frac{\rho^\veps \kappa'(\rho^\veps)+ \kappa(\rho^\veps)}{2}|\nabla_x \rho^\veps|^2   \right| + \left| \veps \kappa(\rho^\veps) \nabla_x \rho^\veps \otimes \nabla_x \rho^\veps \right|\, dx \leq C \Phi_\veps \, .
\]
We infer from \eqref{phiest2} that
\begin{equation*}
 \begin{split}
   \Phi_\veps(t) \leq& 
    C \int_0^t \Phi_\veps(s) ds  + \int \!\!\int_{[0,t)\times\rT^d} \dx\left (\frac{m}{\rho}\right )\veps  \dx ( \rho^\veps \kappa(\rho^\veps) \nabla_x \rho^\veps) \, ds dx\\
 = &   C \int_0^t  \Phi_\veps(s) ds  - \int \!\!\int_{[0,t)\times\rT^d}\nabla_x \left( \dx\left (\frac{m}{\rho}\right )\right)\veps  \rho^\veps \kappa(\rho^\veps) \nabla_x \rho^\veps \, ds dx
 \end{split}
 \end{equation*}
 and obtain that
 \begin{equation}
 \label{phiest3}
 \Phi_\veps(t)
  \leq     C \int_0^t  \Phi_\veps(s) ds  +  \veps  \int \!\!\int_{[0,t)\times\rT^d}\left|\nabla_x \left( \dx\left (\frac{m}{\rho}\right )\right)\right| (\rho^\veps)^2 \kappa(\rho^\veps)\, ds dx.
\end{equation}
Due to $(\rho^\veps)^2 \kappa(\rho^\veps) \lesssim h(\rho^\veps) + \rho^\veps$ equation \eqref{phiest3} implies
\begin{equation}\label{phiest4}
 \begin{split}
   \Phi_\veps(t) \leq  C \int_0^t \Phi_\veps(s) ds  +  \veps C t \left[\int_{\rT^d} h(\rho_0) + \frac{1}{2}\frac{|m_0|^2}{\rho_0} + \frac{\kappa(\rho_0)}{2} |\nabla_x \rho_0|^2 \, dx
   + \int_{\rT^d}  \rho_0 \, dx\right].
 \end{split}
\end{equation}

In turn, applying Gronwall's inequality to \eqref{phiest4} implies
\[ \Phi_\veps(t) \leq \veps C \exp( C t) 
\]
and completes the proof.
\end{proof}


\appendix

\def\cprime{$'$}

\end{document}